\documentclass[12pt]{extarticle}
\usepackage[utf8]{inputenc}
\usepackage[T1]{fontenc}

\usepackage{caption}
\oddsidemargin \evensidemargin
\usepackage{amssymb}
\usepackage{amsmath}
\usepackage{comment}
\usepackage{amsthm}
\usepackage{amssymb}
\usepackage{algorithmicx}
\usepackage{algorithm}
\usepackage{algpseudocode}
\usepackage{mathtools}
\usepackage[margin=1cm]{caption}
\usepackage{subcaption}
\usepackage{mathrsfs}
\usepackage{multirow}
\usepackage[dvipsnames]{xcolor}
\usepackage{tikz,tikz-3dplot,tikz-cd}
\usetikzlibrary{external}
\usepackage{pgfplots}
\usepackage{url}
\usepackage{csquotes}
\usepackage{MnSymbol}
\usepackage[backend=biber,
style=alphabetic,doi=false]{biblatex} 
\usepackage[normalem]{ulem}

\usepackage{hyperref}
\hypersetup{
    colorlinks=true,
    linkcolor=orange!80!black,
    filecolor=magenta,      
    urlcolor=blue,
    citecolor=blue,
}

\definecolor{mycolor1}{rgb}{0.00000,0.44700,0.74100}
\definecolor{mycolor2}{rgb}{0.8500, 0.3250, 0.0980}
\definecolor{mycolor3}{rgb}{0.9290, 0.6940, 0.1250}
\definecolor{mycolor4}{rgb}{0.4940, 0.1840, 0.5560}
\definecolor{mycolor5}{rgb}{0.4660, 0.6740, 0.1880}

\theoremstyle{definition}
\newtheorem{definition}{Definition}[section]
\theoremstyle{theorem}

\newtheorem{proposition}[definition]{Proposition}
\newtheorem{corollary}[definition]{Corollary}
\newtheorem{theorem}[definition]{Theorem}

\newtheorem{lemma}[definition]{Lemma}

\newtheorem*{theorem*}{Theorem}

\theoremstyle{remark}
\newtheorem{remark}[definition]{Remark}
\newtheorem{examplex}[definition]{Example}
\newenvironment{example}
  {\pushQED{\qed}\examplex}
  {\popQED\endexamplex}

\usepackage{mathtools} 
\mathtoolsset{showonlyrefs,showmanualtags}
\numberwithin{equation}{section}

\newcommand{\R}{\mathbb{R}}

\newcommand{\C}{\mathbb{C}}

\newcommand{\Q}{\mathbb{Q}}
\newcommand{\PP}{\mathbb{P}}

\newcommand{\im}{\operatorname{im}}

\newcommand{\wt}{\widetilde}

\DeclareMathOperator{\gr}{Gr}
\newcommand{\grk}{\operatorname{Gr}(k,k+2)}

\newcommand{\aA}{\mathcal{A}}
\newcommand{\CH}{\operatorname{Ch}}
\newcommand{\crank}{\operatorname{cr}}
\newcommand{\sing}{\operatorname{Sing}}
\newcommand{\open}{\operatorname{int}}
\newcommand{\codim}{\operatorname{codim}}
\newcommand{\odd}[1]{([#1])}

\title{Positive Genus Pairs from Amplituhedra}
\author{Joris Koefler, Dmitrii Pavlov, Rainer Sinn}
\date{}

\begin{document}

\maketitle
 
\begin{abstract}
A main conjecture in the field of Positive Geometry states that amplituhedra, which are certain semi-algebraic sets in the Grassmannian, are positive geometries. It is motivated by examples showing that the canonical forms of certain amplituhedra compute scattering amplitudes in particle physics. Beyond a small number of special cases, this conjecture is still open. In recent work, Brown and Dupont introduced a new framework, based on mixed Hodge theory, connecting canonical forms and de Rham cohomology via genus zero pairs.
We give short proofs that the amplituhedron gives rise to a genus zero pair in the cases when it is known to be a positive geometry. 
However, in the general case we show that amplituhedra inside the Grassmannian give rise to pairs of strictly positive genus.
We provide an explicit example of a genus one pair arising from a positive geometry in projective space, showing that having genus zero is not a necessary condition to be a positive geometry.
Finally, we show that this positive geometry still gives rise to a genus zero pair in a different ambient variety. 
\end{abstract}

\section{Introduction}
The \emph{amplituhedron} $\mathcal{A}_{n,k,m}(Z)$, introduced by Arkani-Hamed and Trnka in \cite{Amplituhedron}, is a geometric object of great importance in particle physics, where it is used to calculate scattering amplitudes in $\mathcal{N}=4$ super Yang--Mills theory.
It is defined to be the image of the totally nonnegative Grassmannian $\gr(k,n)_{\geq 0}$ \cite{postnikov} under the linear map $\tilde{Z}:\mathrm{Gr}(k,n)\dashrightarrow\mathrm{Gr}(k,k+m)$ induced by a $n\times (k+m)$ matrix $Z$ all of whose minors are positive. 
Here $n,k,m$ are positive integers such that $k+m\leq n$. 
The case of immediate relevance for physics is $m=4$ where $m$ indicates the dimension of space-time; the case $m=2$ serves as an important toy model.

The discovery of the amplituhedron stimulated active search for further geometric objects that describe scattering amplitudes in different quantum field theories. 
Certain physical properties that an amplitude satisfies can be phrased mathematically in terms of poles and residues of differential forms. 
This observation lead to the introduction of the framework of \emph{positive geometries} in \cite{pos_geom_2017}. 
A central conjecture of \cite{pos_geom_2017} is that the amplituhedron $\mathcal{A}_{n,k,m}(Z)$ is a positive geometry for all values of $n,k,m$ and generic choices of $Z$. 
Despite a vast amount of mathematical work on the amplituhedron that appeared in the recent years (see e.g. \cite{parisi:meq2, lukowski2023meq2, parisi2024magic, lam_face_structure} for $m=2$ and \cite{amplituhedronbcfwtriangulation, meq4} for $m=4$), this conjecture is still wide open. 
The only settled cases are (1) $k=1$: the amplituhedron is a cyclic polytope, see \cite{bernd1988}, \cite[Theorem 1]{lam2024invitation}; (2) $k+m=n$: the amplituhedron is isomorphic to $\mathrm{Gr}_{\geq 0}(k,n)$ \cite[Theorem 9]{lam2024invitation}; (3) $k=m=2$: see \cite[Theorem 1.3]{amplituhedron_adjoint}. 

In recent work, Brown and Dupont \cite{brown-dupont} recast the notion of positive geometries in the language of Deligne's mixed Hodge theory \cite{Deligne1971}. In this context, positive geometries come from pairs of algebraic varieties $(X,Y)$ that have genus zero.
The genus of the pair is defined as the sum of outer Hodge numbers $h^{p,0}$ for $p>0$, of the mixed Hodge structure on the relative cohomology group $H^n(X,Y)$ with $\Q$-coefficients, where $n=\dim X$. A possible definition of positive geometries within this framework is given in \cite[Definition 1.13]{telen2025positivegeometrypolytopespolypols}.

In the framework of \cite{pos_geom_2017} a positive geometry is a pair $(X,X_{\geq 0})$, where $X$ is a complex variety defined over $\mathbb{R}$ and $X_{\geq 0}$ is a semi-algebraic set inside its real points. In the setup of \cite{brown-dupont} one works with pairs $(X,Y)$ of complex varieties such that $Y\subset X$. The way to go between these two approaches is to set $Y=\partial_a X_{\geq 0}$, where $\partial_a$ denotes the algebraic boundary operator (see Section \ref{sec:posGeom} for a definition). 

The notion of positive geometries in \cite{pos_geom_2017} and that in \cite{telen2025positivegeometrypolytopespolypols} based on \cite{brown-dupont} are not equivalent. For instance, a positive geometry in the sense of \cite{pos_geom_2017} is a real algebraic object, while the real numbers play no role in the framework of \cite{brown-dupont}.
It has been an open question if there exist positive geometries in the sense of \cite{pos_geom_2017} that do not give rise to genus zero pairs of \cite{brown-dupont}.
In Section \ref{sec:genus_one_PG} we answer this question by giving an example of a positive geometry which gives rise to a pair of genus one inspired by our results for amplituhedra.

The guiding question of the present paper asks for the relation between the framework in \cite{brown-dupont} and the amplituhedron $\mathcal{A}_{n,k,m}(Z)$ for $m=2$ and $m=4$. We now describe our main results. 
On the positive side, for the cases in which the amplituhedron is known to be a positive geometry, in Section \ref{sec:vanishing_genus} we give a very short proof showing that the pair $(\mathrm{Gr}(k,k+m),\partial_a \mathcal{A}_{n,k,m}(Z))$ has genus zero (Theorem \ref{thm:m=2_k=2_has_genus_zero}). 
Here $\partial_a \mathcal{A}_{n,k,m}(Z)$ denotes the algebraic boundary of the amplituhedron. 
On the negative side, for $k\geq 3$ and large enough $n$ we discover that this pair has strictly positive genus (Theorems \ref{thm:pos_genus} and \ref{thm:genus_m=4}), meaning that the amplituhedron inside the Grassmannian is not, in general, a positive geometry in the sense of \cite[Definition 1.13]{telen2025positivegeometrypolytopespolypols}.
However, in view of the example constructed in Section \ref{sec:genus_one_PG}, we note that positivity of the genus does not immediately disqualify amplituhedra from being positive geometries in the sense of \cite{pos_geom_2017}. This raises the question of how to make the framework of \cite{brown-dupont} match that of \cite{pos_geom_2017}. We comment on a possible approach to this, suggested to us by Cl\'ement Dupont, in Remark \ref{rem:final}.

The remainder of this article is organized as follows. In Section \ref{sec:prelim}, we give necessary definitions and review the main concepts that we will use in the proofs of later sections.
This includes basic notions of positive geometry and mixed Hodge theory.
In Section \ref{sec:vanishing_genus} we show that for the cases $k=1$, $k+m=n$ and $k=m=2$ the pair $(\gr(k,k+m), \partial_a \mathcal{A}_{n,k,m}(Z))$ has genus zero for a generic choice of the matrix $Z$.
In Section \ref{sec:m=2} we prove that for $m=2$ and $n\geq 2(2k-1)$ the pair above has strictly positive genus. Section \ref{sec:m=4} is devoted to an analogous statement for $m=4$.
Finally, in Section \ref{sec:genus_one_PG} we construct an example of a semi-algebraic set $P\subset \PP^3$ which is a positive geometry in the sense of \cite{pos_geom_2017}, but such that the pair $(\PP^3,\partial_aP)$ has genus one.

\section{Preliminaries} \label{sec:prelim}

The \emph{Grassmannian} $\gr(k,n)$ is the variety of $k$-dimensional vector spaces contained in a fixed $n$-dimensional complex vector space. It is an irreducible smooth complex projective algebraic variety of dimension $k(n-k)$. 
We usually consider the Grassmannian in its Pl\"ucker embedding into $\PP^{ \binom{n}{k}-1}$.
This embedding is given by taking a $k\times n$ full rank matrix $\Lambda$, representing a point $[\Lambda]\in \gr(k,n)$ and defining the coordinates on $\PP^{ \binom{n}{k}-1}$ to be the $k\times k$ minors of $\Lambda$.
If it is clear from context, we will often write $\Lambda$ both for the point in the Grassmannian and a matrix representative of it.
It is well known, see e.g. \cite[Corollary 2.33]{degGrkn}, that the degree of the Grassmannian $\gr(k,n)$ in this embedding is
\begin{align}\label{eq:Grkn_deg}
    \deg\gr(k,n)=(k(n-k))! \cdot 
        \frac{0! \cdot 1! \cdot \ldots \cdot (k-1)!}{(n-k)! \cdot (n-k+1)! \cdot \ldots \cdot (n-1)!}.
\end{align}
In the case of $k=2$, that is in the case of the Grassmannian of lines, this simplifies to 
\begin{align}\label{eq:Gr2_deg}
    \deg \gr(2,n)=C_{n-2},
\end{align}
where $C_{n-2}$ is the $(n-2)$-th Catalan number.
The \emph{non-negative Grassmannian} $\gr(k,n)_{\geq 0}$ \cite{postnikov} is the semi-algebraic subset of $\gr(k,n)$, defined as the intersection of $\gr(k,n)$ in its Pl\"ucker embedding into $\PP^{ \binom{n}{k}-1}$ with the non-negative orthant $\PP^{ \binom{n}{k}-1}_{\geq 0}$. 
Equivalently, $\gr(k,n)_{\geq 0}$ consists of all linear spaces represented by real matrices $\Lambda\in \R^{k\times n}$ of rank $k$ whose $k\times k$ minors are all non-negative, up to a global change of sign. 
A (tree-level) amplituhedron depends on the choice of positive integer parameters $k,m$, and $n$, with $n\geq k+m$, and on the choice of a $n\times (k+m)$ matrix $Z$ with positive maximal minors. In this paper we are going to fix $m=2$ or $m=4$. The matrix $Z$ induces a rational map
\[
\widetilde{Z}:\gr(k,n) \dashrightarrow \gr(k,k+m),\quad V\mapsto VZ.
\]
By the assumption on $Z$, this map is well-defined on $\gr(k,n)_{\geq 0}$. The image $\wt{Z}(\gr(k,n)_{\geq 0}$ of the non-negative Grassmannian under this map is called the \emph{(tree-level) amplituhedron} $\aA_{n,k,m}(Z)$, first defined in \cite{Amplituhedron}. 
Henceforth, we will often suppress writing the dependence on $Z$, that is $\aA_{n,k,m}=\aA_{n,k,m}(Z)$.
The amplituhedron is a regular semi-algebraic set in $\gr(k,k+m)$, i.e. it is the Euclidean closure of its interior. In particular, its Euclidean boundary is $\partial\aA_{n,k,m}=\aA_{n,k,m}\setminus\open \aA_{n,k,m}$.
Then, we define the \emph{algebraic boundary} $\partial_a\aA_{n,k,m}\subset\gr(k,k+m)$ as the Zariski closure of the Euclidean boundary $\partial\aA_{n,k,m}$.

The amplituhedron, or rather its algebraic boundary, is the central object of this article.
Therefore, we recall the descriptions of the algebraic boundary of the amplituhedron $\aA_{n,k,m}$ for $m=2$ and $m=4$.
These commonly use the notion of \emph{twistor coordinates}, see e.g. \cite{hodges2013eliminating} or \cite[Definition 4.5]{williams2023positive}. The twistor coordinate $\langle Yi_1\dots i_m\rangle$ (or $\langle YZ_{i_1}\dots Z_{i_m}\rangle$) is defined as the determinant of the $(k+m)\times (k+m)$ matrix obtained by stacking a matrix representing $Y\in\gr(k,k+m)$ together with the rows of $Z$ labeled by $i_1,\ldots, i_m\in [n]$.
A twistor coorindate can be regarded as a function on $\gr(k,k+m)$ and its vanishing locus is a hyperplane section of $\gr(k,k+m)$. 
It will be convenient for an index $i>n$, to take it modulo $n$, such that $i \mod n \in [n]$, in which case we say that the indices are \emph{read cyclically}.
We now record the algebraic boundary of $\partial_a A_{n,k,2}$ as in \cite[Proposition 3.1]{amplituhedron_adjoint}.

\begin{proposition}\label{prop:m=2alg_bd}
    For $m=2$ the algebraic boundary $\partial_a \aA_{n,k,2}$ of $\aA_{n,k,2}$, with $n \geq k+2$,
    is the union of the following hyperplane sections of $\gr(k,k+2)$:
    \[
    \langle YZ_i Z_{i+1} \rangle = 0,
    \]
    where $Y\in \gr(k,k+2)$ and $i\in [n]$, with indices read cyclically.
\end{proposition}

We will now give a simple interpretation of the boundary components. For this, we need one more definition. Let $\Lambda$ be a linear subspace in $\PP^{k+1}$ of dimension $k-d$ and $X$ an irreducible projective variety of dimension $d$. Then, the \emph{Chow hypersurface} $\CH(X)$ of $X$ is 
    \begin{align*}
        \CH(X) = \overline{\{\Lambda\mid X\cap \Lambda\neq \varnothing\}}\subset\gr(k-d+1,k+2).
    \end{align*}

\begin{remark}
    One way of thinking about the algebraic boundary of $\aA_{n,k,2}$ is that it is characterized by the cycle of lines $\operatorname{span}_{\C}(Z_i,Z_{i+1})\subset \PP^{k+1}$ for $i\in [n]$, with indices read cyclically.
    That is, $\partial_a \aA_{n,k,2}$ consists of the Chow hypersurfaces of these lines. This will be important later in the proofs of Section \ref{sec:m=2}.
\end{remark}

Similarly, the algebraic boundary for $m=4$ was conjectured in \cite[Section 5]{arkani2018unwinding} and eventually proven in \cite[Corollary 8.8]{amplituhedronbcfwtriangulation} to be the following.
\begin{proposition}\label{prop:m=4_alg_boundary}
    For $m=4$ the algebraic boundary $\partial_a \aA_{n,k,4}$ of $\aA_{n,k,4}$, with $n \geq k+4$, is the union of the following
    hyperplane sections of $\gr(k,k+4)$:
    \[
    \langle YZ_i Z_{i+1} Z_j Z_{j+1}\rangle = 0,
    \]
    where $Y\in\gr(k,k+4)$ and $\{i,i+1\}\cap\{j,j+1\}=\varnothing$, with $i,j\in [n]$ read cyclically.
\end{proposition}

\subsection{A pinch of Positive Geometry}\label{sec:posGeom}

We recall the definition of \emph{positive geometries} in the sense of \cite{pos_geom_2017}, see also \cite{lam2024invitation}.
Let $X$ be a $d$-dimensional irreducible complex projective variety defined over $\mathbb{R}$. 
Let $X_{\geq 0}$ be a $d$-dimensional semi-algebraic subset of the real points $X(\mathbb{R})$ of $X$. 
We require that the interior $X_{>0}$ of $X_{\geq 0}$ (with respect to the analytic topology on $X(\mathbb{R})$) 
is an oriented manifold, such that its (analytic) closure is $X_{\geq 0}$. 
The algebraic boundary $\partial_a X_{\geq 0}$ is the Zariski closure in $X$ of $X_{\geq 0} \setminus X_{>0}$. 
The irreducible components of $\partial_a X_{\geq 0}$ are prime divisors $D_1,\ldots,D_r$ on~$X$. We define $(D_i)_{\geq 0}$ as the closure in $D_i(\mathbb{R})$ of the interior of $D_i \cap (X_{\geq 0}\setminus X_{>0})$.

\begin{definition}\label{def:pos_geom_lam}
The pair $(X, X_{\geq 0})$ is called a \emph{positive geometry} if there exists a unique rational $d$-form $\Omega(X_{\geq 0})$ on $X$, called the \emph{canonical form}, satisfying the following axioms.
\begin{itemize}
    \item If $d > 0$, the form $\Omega(X_{\geq 0})$ has poles only along $D_1,\ldots,D_r$. 
    \item For each component $D_i$ of the algebraic boundary, the pair $(D_i,(D_i)_{\geq 0})$ is a positive geometry; $\Omega(X_{\geq 0})$ has a simple pole along $D_i$ 
    and the Poincaré residue $\operatorname{Res}_{D_i} \Omega(X_{\geq 0})$ equals the canonical form of 
    $(D_i,(D_i)_{\geq 0})$. Here $D_{i,> 0}$ is an oriented manifold, 
    with orientation induced by that of $X_{>0}$. 
    \item If $d=0$, $X=X_{\geq 0}$ is a point and $\Omega(X_{\geq 0}) = \pm 1$. 
    The sign determines the orientation on this 0-dimensional manifold.
\end{itemize}
\end{definition}

\begin{remark} \label{rem:holomorph}
Note that if the variety $X$ has non-zero global holomorphic top-forms, the pair $(X, X_{\geq 0})$ cannot be a positive geometry: adding such a form to a canonical form would produce another canonical form, contradicting the uniqueness requirement.
Moreover, if the canonical form is not unique, then the difference between any two canonical forms is a global holomorphic top-form on $X$: both canonical forms are holomorphic on $X\setminus\partial_{a}X_{\geq 0}$, and their difference has zero residues along each component of $\partial_{a}X_{\geq 0}$.
\end{remark}

Lastly, the \emph{residual arrangement} $\mathcal{R}(X_{\geq0})\subset X$ of a semi-algebraic set $X_{\geq 0}\subset X$ is defined as the union of the irreducible components of the singular locus of its algebraic boundary $\partial_aX_{\geq0}$ which do not intersect the Euclidean boundary $\partial X_{\geq 0}$.
If there is a unique polynomial of lowest degree interpolating $\mathcal{R}(X_{\geq0})$, then we call it the \emph{adjoint} of $X_{\geq 0}$, denoted $\operatorname{adj}(X_{\geq 0})$.

\subsection{A hint of Hodge Theory}

In the recent paper \cite{brown-dupont} an alternative framework for positive geometries, based on mixed Hodge theory, was suggested. We briefly recall the essential notions.
Let $X$ be a smooth complex variety of dimension $n$.
In this work we exclusively work with singular cohomology groups of $X$ with rational coefficients, denoted by $H^k(X)$.
There is a natural decomposition of the cohomology groups on $X$, first identified by Hodge in \cite{Hodge}:
\begin{align}\label{eq:pure_HS_decomp}
    H^k(X)=\bigoplus_{p+q=k}H^k(X)^{p,q}=\bigoplus_{p+q=k}H^q(X,\Omega_X^p),
\end{align}
where $\Omega_X^p$ is the \emph{sheaf of algebraic $p$-forms} on $X$, which is defined by taking the $p$-th exterior power of the sheaf of regular functions $\mathcal{O}_X$ on $X$.
This is the \emph{Hodge decomposition}, and it gives $X$, or rather its cohomology groups, a \emph{pure Hodge structure of weight $k$}.

Singular varieties can be endowed with a \emph{mixed Hodge structure}, which roughly speaking allows for pure Hodge structures of different weights to be combined. We use the definition as in \cite[Definition 1.5]{brown-dupont}.
\begin{definition}
    A \emph{mixed Hodge structure} is the data of a finite-dimensional $\mathbb{Q}$-vector space $H$ together with
    \begin{itemize}
        \item an increasing filtration $W$ on $H$, called the \emph{weight filtration},
        \item a decreasing filtration $F$ on the complexification $H_{\mathbb{C}}$, called the \emph{Hodge filtration},
    \end{itemize}
    such that for every integer $w$, the filtration induced by $F$ on
    \[
    \operatorname{gr}^W_w H := W_w H / W_{w-1} H
    \]
    defines a pure Hodge structure of weight $w$.
    For $w=p+q$ record its Hodge filtrated parts by
    \[
    (\operatorname{gr}^W_w H)^{p,q}=\operatorname{gr}_F^p(\operatorname{gr}^W_w H)=\frac{F^p(\operatorname{gr}^W_w H)}{F^{p+1}(\operatorname{gr}^W_w H)}=F^p(\operatorname{gr}^W_w H)\cap \overline{F^q(\operatorname{gr}^W_w H)}.
    \]
    The corresponding \emph{Hodge numbers} are the complex dimensions
    \[
    h^{p,q} = \dim H^{p,q} = h^{q,p}, 
    \quad \text{where } H^{p,q} := \bigl(\operatorname{gr}^W_{p+q} H\bigr)^{p,q}.
    \]
\end{definition}
\begin{remark}
    In our setting, we will exclusively deal with mixed Hodge structures on the $k$-th cohomology group of a complex algebraic variety $X$. Then, for the Hodge numbers we explicitly write
    \begin{align}\label{eq:Hodge_numbers}
        h^{p,q}(H^k(X)) = \dim\left(\operatorname{gr}_F^p\bigl(\operatorname{gr}^W_{p+q} H^{k}(X)\bigr)\right).
    \end{align}
\end{remark}
Now, let $Y$ be a closed subset of an $n$-dimensional complex algebraic variety $X$, and let $H^k(X,Y)$ denote the $k$-th relative cohomology group.
The first central notion that allows us to use this framework to study positive geometries is the genus of the pair $(X,Y)$.
We define the \emph{genus} $g(X,Y)$ of the pair to be
\[
    g(X,Y) =\sum_{p>0}h^{p,0}(H^n(X,Y)).
\]
We write $g(X)=g(X,\varnothing)$ for the trivial pair $(X,\varnothing)$.
This definition is equivalent to writing $g(X,Y)=\sum_{p>0}h^{-p,0}(H_n(X,Y))$ essentially by Poincaré duality, see \cite[Equation (16)]{brown-dupont}.
Secondly, denote by $\Omega_{\operatorname{log}}^n(X\setminus Y)$ the space of  global holomorphic differential $n$-forms on $X\setminus Y$ with logarithmic poles along $Y$.
If $X\setminus Y$ is smooth and $g(X,Y)=0$ we have an surjective morphism of $\Q$-vector spaces
\begin{align}\label{eq:canForm_map}
    \operatorname{\Omega}:H_n(X,Y)\twoheadrightarrow\Omega_{\operatorname{log}}^n(X\setminus Y),\quad \sigma\mapsto \Omega(\sigma)=\omega_\sigma,
\end{align}
which associates to a $n$-cycle $\sigma$ in $X$ with boundary in $Y$, a logarithmic differential $n$-form $\Omega(\sigma)$ on $X\setminus Y$.
In this sense, one way to define a positive geometry is as follows, see \cite{telen2025positivegeometrypolytopespolypols}.
\begin{definition}\label{def:posgeom_BD}
    Let $X$ be an $n$-dimensional complex algebraic variety, and $Y$ a divisor, such that the complement $X\setminus Y$ is smooth.
    Further, let the pair $(X,Y)$ have vanishing genus, i.e. $g(X,Y)=0$, and non-vanishing combinatorial rank (see \cite[Definition 2.2]{brown-dupont}) $\crank(X,Y)>0$.
    Then, a \emph{positive geometry} is a relative homology class $\sigma \in H_n(X,Y)$. The \emph{canonical form} of $\sigma$ is the image $\Omega(\sigma)\in\Omega_{\operatorname{log}}^n(X\setminus Y)$ under the map in \eqref{eq:canForm_map}. 
\end{definition}

\begin{remark}
    In several examples, such as hyperplane arrangements, see \cite[Section 6]{brown-dupont}, this notion of canonical form coincides with that of Definition \ref{def:pos_geom_lam}.
    However, this definition is certainly more broad, as it does not have any recursive conditions, see \cite[Remark 2.4.3]{brown-dupont}.
    On the other hand, it is not immediately clear whether being a positive geometry in the sense of Definition \ref{def:pos_geom_lam} implies being a positive geometry in the sense of Definition \ref{def:posgeom_BD}.
    In Section \ref{sec:genus_one_PG} we show that this is not the case.
\end{remark}
    If a mixed Hodge structure $H$ satisfies $H^{p,q}=0$ for all $p\neq q$, then we say $H$ is of \emph{Tate type}.
    Let $a$ be an integer, then denote the \emph{twisted Hodge structure} by $H(-a)$, where $H(-a)^{p,q}:=H^{p+a,q+a}$.
    Additionally, we say a map $f:H_1\rightarrow H_2$ between two Hodge structures is a \emph{morphism} of Hodge structures if $f(F_p(H_1))\subset F_{p}H_2$ for all $p$ or equivalently $f(H_1^{p,q})\subset H_2^{p,q}$ for all $p$ and $q$.
\begin{remark}
    Some other sources choose to introduce morphisms between Hodge structures of type $(a,a)$ for an integer $a$, such that $f(F_p(H_1))\subset F_{p+a}H_2$ or equivalently $f(H_1^{p,q})\subset H_2^{p+a,q+a}$ for all $p$ and $q$.
    We choose to work only with morphisms of Hodge structures of type $(0,0)$, and add Tate twists to the Hodge structures if necessary. 
    For example, for the morphism of type $(a,a)$ from above, we write $f(H_1^{p,q})\subset H_2(-a)^{p,q}$ for all $p$ and $q$ and simply say $f$ is a morphism of Hodge structures.
\end{remark}

Throughout the rest of the article, $Y$ will usually be a reducible subvariety of $X$.
Then, in order to understand the Hodge structure on the relative (co-)homology groups of the pair $(X,Y)$, it will be important to understand how intersections of the irreducible components of $Y$ behave homologically.
To that end, we recall a few standard tools from algebraic topology.
Firstly, let $A,B\subset X$ be subvarieties of the complex algebraic variety $X$ such that $X=A\cup B$, then we have the \emph{Mayer--Vietoris} long exact sequence in cohomology:
\begin{align}\label{eq:MV}
    \cdots \to H^{k-1}(A) \oplus H^{k-1}(B)\to H^{k-1}(A\cap B) \to H^k(X) \to H^k(A) \oplus H^k(B) \to \cdots.
\end{align}
Additionally, let $Y$ be a closed subset of $X$ of codimension one. Then we get the long exact sequence in relative cohomology for the pair $(X,Y)$:
\begin{align}\label{eq:relLES}
    \cdots \to H^{k-1}(X) \to H^{k-1}(Y) \to H^k(X,Y) \to H^k(X) \to \cdots.
\end{align}
Both sequences are long exact sequences of mixed Hodge structures, and all of the maps appearing in them are morphisms of Hodge structures, see \cite[Appendix C.23/30]{Dimca_sing_hypersurfaces} and \cite[Proposition 5.46]{MHS}, respectively.
Lastly, we recall the Lefschetz hypersurface theorem. 
\begin{theorem}[{\cite[Theorem 1.23]{Voisin_2003}}]\label{thm:Lefschetz_Voisin}
    Let $X$ be a complex algebriac variety of dimension $n$, and let $Y$ be a hypersurface in $X$, such that the complement $X\setminus Y$ is smooth. Then, we get an isomorphism
    \begin{align}\label{eq:Lefschetz}
        H^k(X)\cong H^{k}(Y)\quad \text{ for all } k\leq n-2,
    \end{align}
    and an injection $H^{n-1}(X)\hookrightarrow H^{n-1}(Y)$.
\end{theorem}

\section{Amplituhedra that make genus zero pairs}\label{sec:vanishing_genus}

In this section we consider specific choices of the parameters $n$, $k$, and $m$ for which the amplituhedron $\aA_{n,k,m}$ is known to be a positive geometry in the sense of \cite{pos_geom_2017}. We show that for these choices the pair $(\mathrm{Gr}(k,k+m),\partial_a \aA_{n,k,m})$ has genus zero, and thus fits nicely into the framework proposed in \cite{brown-dupont}.
We collect the results in the following theorem. 

\begin{theorem}\label{thm:amplituhedra_genus_zero}
    Let $Z$ be a generic $n\times (k+m)$ matrix with positive maximal minors.
    Then, for any choice of integers $k$, $m$, and $n\geq k+m$, satisfying either $k=1$, $k+m=n$, or $k=m=2$, the boundary divisor $\partial_a\aA_{n,k,m}\subset \gr(k,k+m)$ of the amplituhedron makes a genus zero pair with the ambient Grassmannian.
    That is, 
    \[
    g(\gr(k,k+m),\partial_a\aA_{n,k,m})=0
    \]
    for the choices of $k,n$ and $m$ as above.
\end{theorem}

We cover the cases of the theorem above individually.
The first case is the non-negative Grassmannian, corresponding to the choice $k+m=n$.

\begin{proposition}\label{prop:nonnegGrass_genus_zero}
    The non-negative Grassmannian $\gr(k,n)_{\geq 0}$ gives rise to a genus zero pair. That is, $g(\gr(k,n),\partial_a\gr(k,n)_{\geq 0})=0$.
\end{proposition}
\begin{proof}
    By \cite[Corollary 4.13]{galashin_positroid_hodge} the cohomology groups of open positroid varieties in $\gr(k,n)_{\geq 0}$ are pure of Tate type.
    In particular, this is true for the top-dimensional open positroid variety $\mathring \Pi$. Its intersection with the non-negative Grassmannian $\mathring\Pi\cap \gr(k,n)_{\geq 0}$, has an algebraic boundary coinciding with $\partial_a\gr(k,n)_{\geq 0}$.
    Therefore, $\mathring \Pi=\gr(k,n)\setminus\partial_a\gr(k,n)_{\geq 0} $.
    Then, since $\mathring \Pi$ is smooth, it follows from Poincaré duality in \cite[Equation (16)]{brown-dupont} that 
    \[
    H_n(\gr(k,n),\partial_a\gr(k,n)_{\geq 0})\cong H^n(\gr(k,n)\setminus\partial_a\gr(k,n)_{\geq 0})(n)=H^n(\mathring\Pi)(n).
    \]
    Since $H^n(\mathring \Pi)$ has pure Hodge structure of Tate type and weight $2n$, the relative homology group $H_n(\gr(k,n),\partial_a\gr(k,n)_{\geq 0})$ must have a pure Hodge structure of weight $0$ and Tate type, in particular implying $g(\gr(k,n)\partial_a\gr(k,n)_{\geq 0})=0$.
\end{proof}

The second case is the amplituhedron $\aA_{n,2,2}$.
\begin{theorem}\label{thm:m=2_k=2_has_genus_zero}
    The amplituhedron $\aA_{n,2,2}\subset \gr(2,4)$ gives rise to a genus zero pair, i.e. $g(\gr(2,4),\partial_a\aA_{n,2,2})=0$.
\end{theorem}
\begin{proof}
    The Grassmannian $\gr(2,4)$ is a smooth quadratic hypersurface in its Pl\"ucker embedding into $\PP^5$.
    Moreover, the irreducible components $\mathcal D_i$ for $i\in [n]$ of the algebraic boundary $\mathcal D=\partial_a\aA_{n,2,2}=\bigcup_{i=1}^n\mathcal D_i\subset \gr(2,4)$ can be extended to hyperplanes $D_i$ in the ambient $\PP^5$.
    We also set $D=\bigcup_{i=1}^nD_i$ to be the union of these hyperplanes.
    Now, since the complement $\PP^5\setminus D$ is smooth and affine, we get 
    \begin{align}\label{eq:genera_Ranestad_amplit}
    g(\gr(2,4), \mathcal{D}) = g(\PP^5,D\cup\gr(2,4))-g(\PP^5,D),
    \end{align}
    by \cite[Corollary 3.12]{brown-dupont}.
    Moreover, since $H^{4}(\PP^5)^{p,0}=0$ for all $p\geq 1$, \cite[Corollary 3.11]{brown-dupont} implies that
    \begin{align}\label{eq:genera_Ranestad_amplit_help_1}
    g(\PP^5,D)=g(\PP^5)+g(D),
    \end{align}
    and
    \begin{align}\label{eq:genera_Ranestad_amplit_help_2}
        g(\PP^5,D\cup\gr(2,4))=g(\PP^5)+g(D\cup\gr(2,4)).
    \end{align}
    Using \cite[Theorem 3.26]{brown-dupont} for $g(D)$ and $g(D\cup\gr(2,4))$, as well as the fact that $h^{p,0}(H^5(\PP^5))=0$, we see that both genera vanish in the right-hand sides of \eqref{eq:genera_Ranestad_amplit_help_1} and \eqref{eq:genera_Ranestad_amplit_help_2}.
    Therefore, we obtain $g(\gr(2,4), \mathcal D)=0$ from \eqref{eq:genera_Ranestad_amplit}, as had to be shown.
\end{proof}

We collect these statments to prove the main theorem of this section.
\begin{proof}[Proof of Theorem \ref{thm:amplituhedra_genus_zero}]
    For $k=1$, the amplituhedron $\aA_{n,1,m}\subset \gr(1,k+m)\cong \PP^{k+m-1}$ is a cyclic polytope, see \cite[Theorem 1]{lam2024invitation}. In this case the claim follows from \cite[Section 6]{brown-dupont}.
    For the case $k+m=n$, the map $\tilde{Z}:\gr(k,n)_{\geq 0}\rightarrow \gr(k,n)$ is a linear isomorphism and thus the amplituhedron is isomorphic to the non-negative Grassmannian $\gr(k,n)_{\geq 0}$.
    Therefore, the claim is Proposition \ref{prop:nonnegGrass_genus_zero}.
    Lastly, the case of $k=m=2$ is covered in Theorem \ref{thm:m=2_k=2_has_genus_zero}.
\end{proof}

\section{The $m=2$ amplituhedra $\aA_{n,k,2}$ make positive genus pairs} \label{sec:m=2}

Let $m=2$, $k\geq 3$ and $n\geq2(2k-1)$.
The main goal of this section is to prove that the algebraic boundary $\partial_a\aA_{n,k,2}$ and the Grassmannian $\grk$ form a positive genus pair.
To that end, let the $n\times (k+2)$ matrix $Z$ with positive maximal minors, which defines the amplituhedron $\aA_{n,k,2}=\aA_{n,k,2}(Z)$, be generic.
Here, generic means we pick $Z$ in the open subset of matrices of this size, such that Kleiman's transversality theorem \cite[Theorem 1.7]{3264andallthat} applies in the situations appearing later in this section, see Lemmas \ref{lem:Kleimann} and \ref{lem:curve_smooth_irreducible}.
The main result of this section is then the following. 

\begin{theorem}\label{thm:pos_genus}
    Let $n$ and $k$ be positive integers such that $n\geq 2(2k-1)$. Further let $Z$ be generic and have positive maximal minors, 
    and $\partial_a\aA_{n,k,2}(Z)\subset\gr(k,k+2)$ be the algebraic boundary of the amplituhedron.
    Then, the genus of the pair $(\gr(k,k+2),\partial_a\aA_{n,k,2}(Z))$ is at least $1+\frac{k-3}{2}C_k$, where $C_k$ is the $k$-th Catalan number. In particular, for $k\geq 3$ it is positive. 
\end{theorem}

\begin{remark}
    The lower bound from Theorem \ref{thm:pos_genus} is not tight. The number $1+\frac{k-3}{2}C_k$ is the genus of a single curve in the residual arrangement of the amplituhedron whose cohomology injects into the relative cohomology of the pair $(\gr(k,k+2),\partial_a\aA_{n,k,2}(Z))$. The residual arrangement in general contains multiple curves of positive genus. Moreover, one could in principle have non-zero contributions to the genus from higher-dimensional components of the residual arrangement. However, since our primary goal is to show that the genus of the pair is non-zero rather than to compute it exactly, we do not aim to improve this bound.
\end{remark}

First, let us make a note of the strategy we are going to employ in this section.
Since the Grassmannian is a smooth projective variety, we have $g(\gr(k,k+2),\partial_a\aA_{n,k,2})=g(\partial_a\aA_{n,k,2})$.
Therefore, in order to compute the genus of the pair, we need to understand the algebraic boundary of the amplituhedron.
Moreover, the boundary is reducible and hence its cohomology is essentially determined by multiple intersections of its components.
Therefore, we start by establishing a series of results that will help us compute the cohomology groups of intersections of boundary divisors of the amplituhedron.
Firstly, we show that dense subsets of global complete intersections inside the Grassmannian are local complete intersections.
\begin{lemma}\label{lem:open_lci}
    Let $H$ be a (not necessarily smooth) global complete intersection in $\gr(k,n)$.
    Then, any open subvariety $Y\subset H$ obtained as the complement of a divisor on $H$ is a local complete intersection.
\end{lemma}
\begin{proof}
    By assumption the coordinate ring $\C[H]$, as a finite-dimensional $\C$-algebra, is a (global) complete intersection. Any localization $\C[H]_g$ for $g\in C[H]$ is a complete intersection by \cite[Lemma 10.135.2]{stacks-project_lci}, so in particular a local complete intersection.
    By assumption $Y$ is the complement of a divisor in $H$, therefore we can find $f\in \C[\gr(k,n)]$ such that $f$ is not identically zero on $H$ and $\mathbb{V}(f)\cap H=Y$.
    Then, $\operatorname{Spec}(\C[H]_f)_0\cong Y$ is a local complete intersection.
\end{proof}

The following is a direct consequence of the Lefschetz type theorem \cite[Theorem B]{popa_lefschetz}.
\begin{lemma}\label{lem:CI_Hodge}
    Let $X$ be a (not necessarily smooth) global complete intersection of hyperplane sections in the Grassmannian $\gr(k,n)$.
    Then, the $\ell$-th cohomology group $H^\ell(X)$ is isomorphic to $H^\ell(\gr(k,n))$ for all $\ell\leq \dim X-1$. In particular, the Hodge structure on $H^\ell(X)$ is pure of Tate type of weight $\ell$ for all $\ell\leq \dim X -1$.
\end{lemma}
\begin{proof}
    Since $X$ is a global complete intersection of linear spaces in $\gr(k,n)$, say of codimension $r$, we can find hyperplane sections $H_1,\ldots,H_r\subset \gr(k,n)$ such that $X=H_1\cap\ldots \cap H_r$. 
    We also write $H_{[j]}=\bigcap_{i\in [j]}D_i$ for any $j\in [r]$.
    We prove the claim by induction on the codimension $r$ of $X$.
    The case $r=0$ is tautological.
    If $r =1$, we know that $\gr(k,n)\setminus X$ is smooth, and the claim follows by the Lefschetz hyperplane theorem, see Theorem \ref{thm:Lefschetz_Voisin}.
    
    Now let $r$ be arbitrary, and $X=H_{[r]}$.
    So suppose $H^\ell(H_{[r-1]})\cong H^\ell(\gr(k,n))$ for $\ell<n-r+1$.
    Next, note that $X \subset H_{[r-1]}\subset \gr(k,n)$ is a very ample effective Cartier divisor, since the Grassmannian $\gr(k,n)$ has an embedding into projective space  -- the usual Pl\"ucker embedding.
    Moreover, since $H_{[r-1]}$ is a global complete intersection by definition, it follows from Lemma \ref{lem:open_lci} that the complement $U=H_{[r-1]}\setminus X$ is a local complete intersection, in particular $U$ has vanishing local cohomology defect by \cite[Example 2.5]{popa_lefschetz}.
    Then, \cite[Theorem B]{popa_lefschetz} applies, showing that $H^\ell(X)\cong H^\ell(H_{[r-1]})\cong H^\ell(\gr(k,n))$ for all $\ell<\dim X=n-r$, finishing the proof.
\end{proof}

As a corollary, we obtain a result about Hodge structures on divisors in complete intersections inside the Grassmannian.

\begin{corollary}\label{cor:schubert_divisors_cohomology}
    Let $X$ be a (not necessarily smooth) $d$-dimensional global complete intersection of hyperplane sections in the Grassmannian $\gr(k,n)$. 
    Further, let $r$ be a positive integer and $D=D_1\cup\ldots\cup D_r$ be a union of $r$ Schubert hyperplane sections $D_i$ in $\gr(k,n)$.
    Then, the $\ell$-th cohomology group $H^\ell(X\cap D)$ is isomorphic to $H^\ell(\gr(k,n))$ for all $\ell\leq d-2$, in particular its Hodge structure is pure of weight $\ell$ and Tate type.
\end{corollary}
\begin{proof}
    By assumption $X$ is a global complete intersection, hence it follows from Lemma \ref{lem:open_lci}, that $U=X\setminus (D\cap X)$ is a local complete intersection.
    In particular, its local cohomology defect vanishes by \cite[Example 2.5]{popa_lefschetz}.
    Next, the line bundle $\mathcal{O}_{\grk}(r)\vert_X$ is ample.
    Then, $X\cap D$ is a section of this line bundle and as such ample itself; therefore $X\cap D$ is an ample effective Cartier divisor in $X$.
    Then, $H^{\ell}(X\cap D)\cong H^\ell(X)$ for all $\ell \leq d-2$ by \cite[Theorem B]{popa_lefschetz}.
    The cohomology of $H^\ell(X)$ is isomorphic to $H^\ell(\gr(k,n))$, as computed in Lemma \ref{lem:CI_Hodge}, finishing the proof.
\end{proof}

Next, we show that, in the same setting, we get a very useful lower bound for some Hodge numbers in the Hodge filtrated parts of cohomology groups in the Mayer--Vietoris sequence.
\begin{corollary}\label{cor:injection}
    Let $X$ be a (not necessarily smooth) $d$-dimensional global complete intersection of hyperplane sections in the Grassmannian $\gr(k,n)$. 
    Further, let $D_1,\ldots,D_r$ be Schubert hyperplane sections in $\gr(k,n)$, we let $D=X\cap D_1$ and $R = X\cap (D_2\cup\ldots\cup D_r)$ be the intersection of $X$ with such a Schubert divisor, and a union of Schubert divisors, respectively.
    Then, we get an injection in the Mayer--Vietoris long exact sequence in cohomology
    \[
    H^{\ell-1}(D\cap R)^{p,0}\hookrightarrow H^\ell(D\cup R)^{p,0}
    \]
    for all $\ell\leq d-1$ and $p>0$. In particular, $h^{p,0}( H^{d-2}(D\cap R))\leq h^{p,0}( H^{d-1}(D\cup R))$. 
\end{corollary}
\begin{proof}
    We consider the Mayer--Vietoris long exact sequence in cohomology for the union $D\cup R$, see Equation \eqref{eq:MV}:
    \[
    \cdots \to H^{\ell-1}(D) \oplus H^{\ell-1}(R)\to H^{\ell-1}(D\cap R) \to H^\ell(D\cup R) \to H^\ell(D) \oplus H^\ell(R) \to \cdots.
    \]
    The term $H^{\ell-1}(D) \oplus H^{\ell-1}(R)$, with $D=X\cap D_1$ and $R=X\cap(D_2\cap\ldots\cap D_r)$, admits a pure Hodge structure of Tate type for all $\ell\leq d-1$ by Corollary \ref{cor:schubert_divisors_cohomology}.
    In particular, the Hodge numbers $h^{p,0}$ vanish for all $p>0$.
    Therefore, the claim follows since the Mayer--Vietoris sequence is a long exact sequence of mixed Hodge structures, see \cite[C.30]{Dimca_sing_hypersurfaces}.
\end{proof}

We fix some notation for the remainder of this section.
For $i\in [n]$, let $D_i$ be the irreducible components of the algebraic boundary $\partial_a\aA_{n,k,2}$. In particular, $D_i$ denotes the Chow hypersurface $\CH(L_i)$ of the line $L_i\subset \C^{k+2}$ spanned by the rows $Z_i$ and $Z_{i+1}$ of the matrix $Z$.
Further, for $I\subset [n]$ we write $D_I$ for the intersection $\bigcap_{i\in I}D_i$. Lastly, for any odd $j\in [n]$, we write $D_{\odd j}$ for the intersection of $D_i$ for all odd integers $i\in[j]$.

\begin{lemma}\label{lem:Kleimann}
    Let $I$ be a subset of $[n]$ such that $i\in I\Rightarrow i+1\notin I$.
    Then, the intersection $D_I=\bigcap_{i\in I}D_i$ is a global complete transversal intersection in $\grk$.
    Geometrically speaking, for a set of pairwise skew lines $L_i$ the intersection of the corresponding Schubert divisors $D_i$ is transversal and a global complete  intersection.
    In particular, $D_{\odd j}$ is a transversal and global complete intersection for any odd $j\in [n]$.
\end{lemma}
\begin{proof}
    For the Grassmannian $\grk$ the general linear group $\operatorname{GL}_{\C}(k+2)$ acts transversely on it as an algebraic group.
    By our choice of $I$, and since $Z$ is generic, the lines $L_i$ are in general position.
    Therefore, Kleiman's theorem, \cite[Theorem 1.3]{3264andallthat}, applies, showing that the intersection $\bigcap_{i\in I}D_i$ is generically transversal. This shows the first part of the statement.
    Lastly, notice that if we consider $D_{\odd j}$, then in particular the lines $L_i$ for $i\in([j])$ are skew, and thus in general position, so this is a special case of the above.
    This finishes the proof.
\end{proof}

\begin{lemma}\label{lem:curve_smooth_irreducible}
    Let $k\geq 2$ and $n\geq 2(2k-1)$. Take $I\subset [n]$, with $|I|=2k-1$, and such that the lines $L_i = \operatorname{span}(Z_i,Z_{i+1})$ for $i\in I$ are pairwise skew.
    Then, the intersection $E=\bigcap _{i\in I}\CH(L_i)$ in $\grk$ is a smooth, irreducible and reduced curve.
\end{lemma}
\begin{proof}
    By assumption the lines $L_i$ are skew and since $Z$ is generic the lines are in general position, so Lemma \ref{lem:Kleimann} applies, stating that the dimension of $E$ is indeed one.
    Moreover, it also follows from the same lemma that for all $J\subset I$ the intersections $\bigcap_{j\in J}D_j$ are global complete intersections, and in particular are transversal.
    Therefore, we have 
    \begin{align}\label{eq:singular_locus_E}
        \sing (E)=\bigcup_{i\in I}\left(\sing D_i\cap \bigcap_{j\neq i}D_j\right),
    \end{align}
    since points that are smooth on each component will be smooth points of the intersection by the transversality of the intersection.
    We want to show that the union in \eqref{eq:singular_locus_E} is empty, and in order to do so we consider the following incidence variety:
    \[
    \Sigma_i=\{ (p,[\Lambda])\mid p\in \Lambda  \}\subset L_i\times \grk.
    \] 
    Since this incidence variety $\Sigma_i$ is smooth, see e.g. \cite[Proposition 3.4]{koefler2025takingamplituhedronlimit}, it follows from \cite[Corollary 14.10]{Harris_AG}, that a point $[\Lambda] \in D_i$ is smooth if and only if $[\Lambda] \cap L_i=\{p\}$.
    Conversely, the singular locus $\sing D_i$ is given  by the subset of $D_i\subset \gr(k,k+2)$ of all $k$-planes $[\Lambda]$ that contain the line $L_i$.
    In particular, the codimension of the singular locus $\sing D_i$ inside $D_i$ is $3$, that is $\dim \sing D_i=2k-4$.
    By the genericity assumption on $Z$, the intersection $\sing D_i\cap \bigcap_{j\neq i\in I}D_j$ is also transversal.
    So the dimension count 
    \[
    \dim\sing D_i-\codim \bigcap_{j\neq i}D_i=2k-4-2k+2=-2
    \]
    shows that this intersection is empty.
    Since $D_i$ was arbitrary, by Equation \eqref{eq:singular_locus_E} the curve $E$ must be smooth. And since $E$ is smooth it must be reduced as a scheme, else the Jacobian criterion for smoothness fails.
    
    It remains to show that the curve is irreducible. For a smooth curve it suffices to show that $E$ is connected.
    This follows from \cite[Theorem 1.1]{martinelli2017connectedness} applied to the Pl\"ucker embedding $\grk\hookrightarrow\PP^{\binom{k+2}{2}-1}$.
\end{proof}

\begin{example}\label{ex:A10_3_2} 
    Let $n=10$, $k=3$, $m=2$, and set $Z$ to be the $10\times 5$ matrix obtained by stacking $10$ points of the rational normal curve of degree $4$ on top of each other.
    More succinctly, 
    \[
        Z = \begin{bmatrix}
            1 & 0 & 0 & 0 & 0\\
            1 & 1 & 1^2 & 1^3 & 1^4 \\
            \vdots & \vdots & \vdots & \vdots & \vdots \\
            1 & 10 & 10^2 & 10^3 & 10^4
        \end{bmatrix}\in \R^{10\times 5}.
    \]
    We consider the amplituhedron $\aA_{10,3,2}(Z)$, with the curve $E=D_{\odd{9}}=D_1\cap D_3\cap D_5\cap D_7 \cap D_9$.
    Then, a computation in \texttt{Macaulay2} \cite{M2} yields that the prime ideal defining the curve $E$ in the coordinate ring $\C[\gr(3,5)]$, is generated by the following five polynomials:
    {\footnotesize
    \begin{align} 
        &-p_{234}-p_{235}-p_{245}+p_{345},\\
        &216p_{123} - 180p_{124} + 114p_{134} - 65p_{234} + 36p_{125} + 30p_{135} - 19p_{235} + 6p_{145} - 5p_{245} + p_{345}, \\
        &8000p_{123} - 3600p_{124} + 1220p_{134} - 369p_{234} + 400p_{125} + 180p_{135} - 61p_{235} + 20p_{145} - 9p_{245} + p_{345}, \\
        &74088p_{123} - 22932p_{124} + 5334p_{134} - 1105p_{234} + 1764p_{125} + 546p_{135} - 127p_{235} + 42p_{145} - 13p_{245} + p_{345}, \\
        &373248p_{123} - 88128p_{124} + 15624p_{134} - 2465p_{234} + 5184p_{125} + 1224p_{135} - 217p_{235} + 72p_{145} - 17p_{245} + p_{345}.
    \end{align}
    }\noindent
    These are precisely the generators of the vanishing ideals of the divisors labeled by odd indices $D_1$, $D_3$, $D_5$, $D_7$ and $D_9$.
    We then also verify that $E$ is smooth and has genus one, in alignment with the following lemma. The code for this example is available at \cite{zenodo}.
\end{example}
We compute the (geometric) genus of the curve $E$ from Lemma \ref{lem:curve_smooth_irreducible}.
\begin{lemma}\label{lem:genus_residual_curve}
    Let $E$ be a curve obtained from the $(2k-1)$-fold transversal intersection of $D_i=\CH(L_i)$, where the lines $L_i\subset \PP(\C^{k+2})$ spanned by $Z_i$ and $Z_{i+1}$ are pairwise disjoint.
    Then, for $k\geq 3$ the (geometric) genus of $E$  can be computed as 
    \[
    g(E)=1+\frac{k-3}{2}C_k,
    \]
    where $C_k$ is the $k$-th Catalan number. For $k=1,2$ the genus of the curve $E$ is zero.
\end{lemma}
\begin{proof}
    Let $k=1$. Then, we are in $\gr(1,3)\cong \PP^2$, and the divisors $D_i$ are lines in $\PP^2$. In particular, $E$ is such a line and thus has genus $0$.
    For $k=2$ we get a linear section of $\gr(2,4)$, which is a conic and thus has genus $0$. 
    Now, let $k\geq 3$.
    By construction we can apply Lemmas \ref{lem:Kleimann} and \ref{lem:curve_smooth_irreducible} to $E$, showing that this curve is a smooth irreducible complete intersection in $\grk$.
    Therefore, it follows from \cite[B.7.4]{Fulton1998}, that that the normal bundle $N_E\grk$ of $E\subset\grk$ decomposes as 
    \begin{align}\label{eq:decomp}
    N_E\grk=\bigoplus_{i=1}^{2k-1}N_{D_i}\grk\vert_E,
    \end{align}
    where $N_{D_i}\grk$ are the normal bundles of the divisors $D_i$.
    Since all of the divisors $D_i$ can be extended to hyperplane sections in $\grk$, we can identify the normal bundles $N_{D_i}\grk\vert_E$ with the line bundles $\mathcal O_{\grk}(1)$.
    It then follows from \cite[Proposition 8.20]{Hartshorne1977} that
    \[
    \omega_E\cong \omega_{\grk}\otimes\wedge^{2k-1}N_{E}\grk,
    \]
    where $\omega_E$, and $\omega_{\grk}$ refer to the canonical divisors of $E$ and $\grk$, respectively.
    Then, using the the decomposition in Equation \eqref{eq:decomp}, we identify $\wedge^{2k-1}N_{E}\grk$ with the line bundle $\mathcal{O}_{\grk}(2k-1)$.
    Then, since the canonical divisor $\omega_{\grk}$ on the Grassmannian is isomorphic to $\mathcal{O}_{\grk}(-k-2)$, see \cite[Section 5.7.4]{3264andallthat},
    we get
    \begin{align}\label{eq:can_div_curve}
    \omega_E\cong \mathcal{O}_{\grk}(-k-2)\otimes\mathcal{O}_{\grk}(2k-1)\cong \mathcal{O}_{\grk}(k-3).
    \end{align}
    Therefore, we can now apply the Riemann--Roch theorem, see e.g. \cite[Theorem 3.23]{shafarevich}, so that
    \[
    \deg_E \omega_E=(k-3)[H][E]=2g-2,
    \]
    where $[E]$, and $[H]$ are the classes in the Chow ring $A^\bullet(\grk)$ represented by $E$ and a hyperplane section $H$ in $\grk$, respectively.
    Therefore, we get $g=1+\frac{(k-3)}{2}[H]^{2k}$ since $[E]=[H]^{2k-1}$, as $E$ is a complete intersection by Lemma \ref{lem:Kleimann}.
    Finally, recall from Equation \eqref{eq:Gr2_deg} that the degree of the Grassmannian $\grk$ in its Pl\"ucker embedding is precisely the $k$-th Catalan number $C_k=[H]^{2k}$, and the claim follows.
\end{proof}

We are now finally in a position to prove the main result of this section.

\begin{proof}[Proof of Theorem \ref{thm:pos_genus}]
    Let $\Delta = \partial_a\aA_{n,k,2}=\bigcup_{i=1}^{n}D_i$ be the decomposition of the algebraic boundary of the amplituhedron into irreducible components $D_i\subset\grk$.
    For $I\subset [n]$ we write $D_I$ for the intersection $\bigcap_{i\in I}D_i$.
    We start by considering the following long exact sequence in relative cohomology.
    \[
    \cdots \to H^{2k-1}(\grk)\to H^{2k-1}(\Delta)\to H^{2k}(\grk,\Delta)\to H^{2k}(\grk)\to\cdots.
    \]
    By \cite[Proposition 5.46]{MHS} this is also a long exact sequence of mixed Hodge structures where the appearing maps are morphisms of such structures.
    Since the Hodge structure on the cohomology of the Grassmannian is pure of Tate type, it follows that the $(p,0)$ part of the terms $H^{2k-1}(\grk)$ and $H^{2k}(\grk)$ in the above exact sequence are zero, in turn implying that $h^{p,0}(H^{2k}(\grk,\Delta)=h^{p,0}(H^{2k-1}(\Delta))$ for all $p>0$.
    Now, we proceed by iteratively applying the Mayer--Vietoris long exact sequence to the union and intersections of the $D_i$.
    We consider $D_{\odd j}$ for some odd $j\in [n]$, and let $d$ denote its dimension.
    By convention we set $D_{\odd j}=\grk$ for $j<0$.
    Note the $D_{\odd j}$ are global complete intersections in the Grassmannian by Lemma \ref{lem:Kleimann}.
    Then, let $R_j$ be the union of all $D_i$ with $i\notin\odd j$ intersected with $D_{\odd{j-2}}$, that is $R_j=D_{\odd{j-2}}\cap \bigcup_{i\notin \odd j}D_i$, and $\dim R_j=d$.
    The union $D_{\odd j}\cup R_j$ fits into the Mayer--Vietoris long exact sequence \eqref{eq:MV} as 
    \begin{equation}\label{eq:MV_for_genus_bound}
        H^{d-1}(D_{\odd j}) \oplus H^{d-1}(R_j) \to H^{d-1}(D_{\odd j} \cap R) \to H^{d}(D_{\odd j}\cup R_j) \to H^{d}(D_{\odd j}) \oplus H^{d}(R_j).
    \end{equation}
    Since $D_{\odd j}$ and $R_j$ as divisors in the complete intersection $D_{\odd{j-2}}$ satisfy the assumptions in Corollary \ref{cor:injection}, we get an injection $H^{d-1}(D_{\odd j} \cap R)^{p,0} \hookrightarrow H^{d}(D_{\odd j}\cup R_j)^{p,0}$ and thus the inequality 
\begin{align}\label{eq:dim_inequality}
        h^{p,0}(H^{d-1}(D_{\odd j} \cap R_j))\leq h^{p,0}( H^{d}(D_{\odd j}\cup R_j)).
    \end{align}
    By noting that $D_{\odd j}\cap R_j=D_{\odd{j+2}}\cup R_{j+2}$, which are divisors in $D_{\odd j}$, we can apply Equation \eqref{eq:MV_for_genus_bound} iteratively until $D_{\odd j}$ and $R_j$ become varieties of dimension one.
    Then, $D_{\odd{mk-1}}$ becomes a smooth curve $E$ by Lemma \ref{lem:curve_smooth_irreducible}.
    We get the following Mayer--Vietoris sequence for $E\cup R_{2k-1}$:
    \[
    \ldots\to H^0(E\cap R_{2k-1})\to H^1(E\cup R_{2k-1})\to H^1(E)\oplus H^1(R_{2k-1})\to H^1(E\cap R_{2k-1}) \to\ldots.
    \]
    Note that zeroth-cohomology groups always have pure Hodge structure of weight zero, and the first cohomology group on a union of points vanishes completely, i.e. $H^1(E\cap R_{2k-1})=0$.
    Therefore, we get an isomorphism of vector spaces $H^1(E\cup R_{2k-1})^{p,0}\cong H^1(E)\oplus H^1(R_{2k-1})^{p,0}$. In particular,
    \[
    h^{1,0}(H^1(E\cup R_{2k-1}))\geq h^{1,0}(H^1(E))=g(E)=1+\frac{k-3}{2}C_k,
    \]
    where the last inequality follows from Lemma \ref{lem:genus_residual_curve}.
    Then, since $E\cup R_{2k-1}=D_{\odd{2k-3}}\cap R_{2k-3}$, where $D_{\odd{2k-3}}$ is of dimension two, we get a sequence of inequalities for the Hodge numbers we are interested in:
    \begin{align}
    h^{p,0}(H^{2k-1}(\partial_a\aA_{n,k,2})) \geq \ldots \geq h^{1,0}(H^2(D_{\odd{mk-3}}\cup R_{2k-3})) \geq h^{1,0}(H^1(E\cup R_{2k-1}))\geq g(E)
    \end{align}
    by Equation \eqref{eq:dim_inequality}, and with $g(E)=1+(k-3)C_k/2$.
    Lastly, $h^{p,0}(H^{2k}(\grk,\partial_a\aA_{n,k,2})) = h^{p,0}(H^{2k-1}(\partial_a\aA_{n,k,2}))$ as asserted above. This finishes the proof.
\end{proof}

To conclude the section, we show that curves $E$ stemming from generic intersections of Schubert hyperplane sections, as in Lemma \ref{lem:curve_smooth_irreducible}, are entirely residual, that is to say, they are contained in the residual arrangement $\mathcal{R}(\aA_{n,k,2})$ of the amplituhedron $\aA_{n,k,2}$, see Section \ref{sec:posGeom} for the definition. 

\begin{proposition}\label{prop:residual_curves}
    Let $E$ be a curve as in Lemma \ref{lem:curve_smooth_irreducible}.
    Let $k>1$.
    Then, $E$ is contained in the residual arrangement $\mathcal{R}(\aA_{n,k,2})\subset \grk$ of the amplituhedron $\aA_{n,k,2}$.
    In particular, the curve $E$ does not meet the amplituhedron.
\end{proposition}

To prove this statement we make use of the results in \cite{lam_face_structure}, giving a combinatorial characterization of the boundary face structure of the $m=2$ amplituhedron.
We define the \emph{twistor embedding}
\[
    T_Z:\gr(k,k+m)\dashrightarrow\gr(m,n),\quad X\mapsto  \langle Xi_1\ldots i_m\rangle,
\]
where $X\in\grk$ is given as the rowspan of a $k\times(k+2)$ matrix, and $\langle Xi_1\ldots i_m \rangle$ are twistor coordinates, as defined in Section \ref{sec:prelim}.
By \cite[Lemma 1]{lam_face_structure} the map $T_Z$ is indeed an embedding.
The face stratification of $\aA_{n,k,2}$ can be studied in this embedding and \cite[Theorem 25]{lam_face_structure} shows that it is a subposet of the positroid stratification of $\gr(2,n)$ (see e.g. \cite[Section 3]{lam_face_structure} for a definition).
By these results, we are left with studying rank two positroids $\mathcal N$ on $[n]$, which we describe as follows.
Let $L \subset [n]$ denote the set of all loops of $\mathcal{N}$.
Also, consider a collection of disjoint cyclic intervals $[a_1,b_1], \ldots, [a_r,b_r]$, such that $\operatorname{rank}([a_i,b_i]) = 1$. 
In other words, all elements in $[a_i,b_i]$, that are not loops, are parallel. 
We always assume in this notation that $r$ is taken to be minimal and each $[a_i,b_i]$ is taken as small as possible. 
Then, we write for a rank two matroid $\mathcal{N}$ on $[n]$ 
\[
    \mathcal N = \big(L,\{[a_1,b_1],\ldots,[a_r,b_r]\}\big).
\] 
Then, we also set
\begin{align}\label{eq:S_matroid}
    S = S(\mathcal N) = \bigcup_{i=1}^r [a_i,b_i],
\end{align}
and define
\begin{align}\label{eq:other_matroid_numbers}
    d(\mathcal N) = 2k + r - |S \setminus L| - 2|L|,\quad
    c(\mathcal N) = 2k - d(N),\quad\text{and}\quad
    e(\mathcal N) = r + k - |S \cup L|.
\end{align}
Finally, we also denote the subset of rank two matroids $\mathcal{N}$ on $[n]$, such that $e(\mathcal N)\geq 0$ by $P_{n,k}$.
This is in alignment with \cite[Definition 13]{lam_face_structure} and \cite[Theorem 16]{lam_face_structure}.
We can now proceed with proving the above proposition.

\begin{proof}[Proof of Proposition \ref{prop:residual_curves}]
    The stratification of the boundary of the amplituhedron $T_Z(\aA_{n,k,2})\subset \gr(2,n)$ in its twistor embedding for $m=2$ was completely characterized in \cite{lam_face_structure}.
    More precisely, its residual arrangement is given in \cite[Proposition 38]{lam_face_structure}.
    We proceed by translating the matroidal language used in this reference to our setting, and then show that the conditions of \cite[Proposition 38]{lam_face_structure} are met.
    By assumption $E$ is given as a complete intersection of Schubert hyperplane sections, stemming from Chow hypersurfaces $D_i=\CH(L_i)$ of skew lines $L_i$ in general position for $i\in I$.
    Any such divisor $D_i$ corresponds to a positroid variety $\Pi_{ii+1}$ of codimension one, given by the vanishing of the Pl\"ucker coordinate $p_{ii+1}$ on $\gr(2,n)$ corresponding to $D_i$ under the twistor embedding $T_Z$.
    More succinctly, $T_Z(D_i)= \Pi_{ii+1}\cap \im T_Z\subset \gr(2,n)$.
    Therefore, the matroid $\mathcal{N}$ of rank two on $[n]$, characterizing $\Pi_{ii+1}$, is given by $L=\varnothing\subset [n]$ and $\{[i,i+1]\}$.
    In general then, the matroid $\mathcal{N}$ of rank two on $[n]$, characterizing $E=\Pi_\mathcal N$, is given by $L=\varnothing\subset [n]$, and 
    \[
    A = \{[i,i+1]\mid i\in I\},
    \]
    where $[i,i+1]=\{i,i+1\}$ read cyclically in $[n]$.
    It follows from Equation \eqref{eq:S_matroid}, that $S=\bigcup_{i\in I}[i,i+1]$ with $|S|=2(2k-1)$, as well as $r=|A|=2k-1$.
    Similarly, by Equation \eqref{eq:other_matroid_numbers}, we then get $e(\mathcal N)=r+k-|S\cup L|=2k-1+k-2(2k-1)=-k+1$, and therefore, by \cite[Theorem 16]{lam_face_structure}, the matroid $\mathcal{N}$ is not in $\mathcal{P}_{k,n,2}$ for $k>1$.
    The other conditions in \cite[Proposition 38]{lam_face_structure}, are all satisfied for $E=\Pi_{\mathcal N}$, since $c(\mathcal{N})=2k-(2k+(2k-1)-2(2k-1))=2k-1$ is less than $2k$, and $k-|L|=k\geq 0$, and $d_i=|\{i,i+1\}\setminus L|-1=1<k$ for all $i\leq r=2k-1$.
    This shows the claim.
\end{proof}

\section{The $m=4$ amplituhedra make positive genus pairs}
\label{sec:m=4}

In this section we record analogous results for the genus of the amplituhedron and its boundary divisor as in Section \ref{sec:m=2}, but for the case $m=4$. Since all of the results are derived analogously we only point out the differences in the proofs. The main result is as follows.

\begin{theorem}\label{thm:genus_m=4}
    Let $k\geq 3$ and $n$ be nonnegative integers, such that $n\geq 4(4k-1)$.
    Further let $Z$ be generic and have positive maximal minors, 
    and $\partial_a\aA_{n,k,4}(Z)\subset\gr(k,k+4)$ be the algebraic boundary of the amplituhedron.
    Then, the genus of the pair $(\gr(k,k+4),\partial_a\aA_{n,k,4}(Z))$ is at least $1+\frac{3k-5}{2}\deg(\gr(k,k+4))$, where the degree of the Grassmannian $\gr(k,k+4)$ is 
    \[
        \deg(\gr(k,k+4))=4k! \cdot 
        \frac{0! \cdot 1! \cdot \ldots \cdot (k-1)!}{4! \cdot 5! \cdot \ldots \cdot (k+3)!}.
    \]
\end{theorem}

For the remainder of this section, we will denote the irreducible components of the boundary divisor of $\aA_{n,k,4}$ by $D_{ij}=D_{ii+1jj+1}=\langle Y i (i+1)j(j+1)\rangle\subset\gr(k,k+4)$ with $Y\in \gr(k,k+4)$, see \cite{amplituhedronbcfwtriangulation} and also Proposition \ref{prop:m=4_alg_boundary}.
Furthermore, we define the 4-plane $V_{ij}=V_{ii+1jj+1}=\operatorname{span}_{\C}(Z_i,Z_{i+1},Z_j,Z_{j+1})\subset \C^{k+4}$, where $Z_k$ are the rows of the $Z$ matrix. In particular, we then have $D_{ij}=\CH(V_{ij})$.

\begin{remark}
    The choice of $n\geq 4(4k-1)$ in Theorem \ref{thm:genus_m=4} is simply to ensure that we can pick $4k-1$ many $4$-planes $V_{ij}$ that are pairwise disjoint.
    The reason for the lower bound on the genus of the pair is a higher genus curve appearing in the intersection of boundary divisors, analogously to the main result in Section $\ref{sec:m=2}$.
    Moreover, similarly to Proposition \ref{prop:residual_curves}, we believe that this curve is in the residual arrangement of $\aA_{n,k,4}$, however since the tools used in that proposition are unavailable for the case $m=4$, it is beyond the scope of this article to prove this statement.
\end{remark}

In order to prove Theorem \ref{thm:genus_m=4}, we establish the analogous lemmas as in Section \ref{sec:m=2} that helped us prove its main result.

\begin{lemma}\label{lem:m=4_Kleimann}
    Let $\mathcal{I}$ be a collection of sets $I=\{i,i+1,j,j+1\}$ such that $I\cap I'=\varnothing$ for all $I,I'\in \mathcal{I}$.
    Then, the intersection $D_\mathcal{I}=\bigcap_{I\in \mathcal{I}}D_I$ is a global complete transversal intersection in $\gr(k,k+4)$.
    Geometrically speaking, for a set of pairwise disjoint $4$-planes $V_{ij}$ the intersection of $D_{I}$ is transversal and a global complete intersection.
    In particular, if $|\mathcal{I}|=4k-1$, then $D_\mathcal{I}$ is a smooth irreducible and reduced curve.
\end{lemma}
\begin{proof}
    For $|\mathcal{I}|<4k-1$ the proof is identical to that of Lemma \ref{lem:Kleimann} with $\gr(k,k+4)$ replacing the Grassmannian $\grk$, and $Z$ being generic enough with respect to the planes $V_{ij}$, instead of the lines $L_i$.
    For $|\mathcal{I}|=4k-1$ the proof of Lemma \ref{lem:curve_smooth_irreducible} applies, upon noting that the singular locus of $D_{ij}$ is the subset of $D_{ij}\subset\gr(k,k+4)$ given by $k$-planes meeting $V_{ij}$ in dimension at least two.
\end{proof}

\begin{lemma}\label{lem:m=4_genus_residual_curve}
    Let $\mathcal{I}$ be a collection of $4k-1$ sets $I=\{i,i+1,j,j+1\}$ such that $I\cap I'=\varnothing$ for all $I,I'\in \mathcal{I}$.
    Then, let $E$ be a curve obtained from the $(4k-1)$-fold intersection $D_{\mathcal{I}}=\bigcap_{I\in \mathcal{I}}D_I$ of boundary divisors $D_{ij}$.
    Then, for $k\geq 2$ the (geometric) genus of $E$  can be computed as 
    \[
    g(E)=1+\frac{3k-5}{2}\deg(\gr(k,k+4)).
    \]
    For $k=1$ the genus of the curve $E$ is zero.
\end{lemma}
\begin{proof}
    The proof is identical to that of Lemma \ref{lem:genus_residual_curve} upon noting the following.
    Firstly, the canonical divisor of the curve $E$ obtained from $4k-1$ intersections of Schubert hyperplane sections, is given as $\omega_E\cong \mathcal{O}_{\gr(k,k+4}(-k-4)\otimes \mathcal{O}_{\gr(k,k+4)}(4k-1)\cong \mathcal{O}_{\gr(k,k+4)}(3k-5)$, as in \eqref{eq:can_div_curve}.
    Secondly, the degree of $\gr(k,k+4)$ now given by \eqref{eq:Grkn_deg} for $n=k+4$.
    Then, applying the Riemann--Roch theorem we obtain the result.
\end{proof}

We are now in a position to prove the main result of this section.

\begin{proof}[Proof of Theorem \ref{thm:genus_m=4}]
    Note that Lemmas \ref{lem:open_lci}, \ref{lem:CI_Hodge}, and Corollary \ref{cor:schubert_divisors_cohomology} apply to any Grassmannian and union of Schubert hyperplane sections.
    Therefore, in particular, they apply to $D=\partial_a\aA_{n,k,4}\subset\gr(k,k+4)$.
    Hence, the proof of Theorem \ref{thm:pos_genus} of applies here, where we replace $D_{\odd{j}}$ for any odd $j\in [n]$ with an appropriate subset of $\mathcal{I}$, defined as in Lemma \ref{lem:m=4_genus_residual_curve}.
\end{proof}

\section{A genus one positive geometry}
\label{sec:genus_one_PG}

In this section we present an explicit example of a semi-algebraic set $P \subset \mathbb{P}^3$ such that $(\mathbb{P}^3, P)$ is a positive geometry in the sense of \cite{pos_geom_2017}, but the pair $(\mathbb{P}^3, \partial_a P)$ has genus one. This demonstrates that being a genus zero pair is not a necessary condition to be a positive geometry, and that the amplituhedron might still be one. The code to reproduce the computations of this sections is available at \cite{zenodo}.
Our example $P$, seen in Figure \ref{fig:3D_DelPezzo}, is combinatorially a full dimensional cube in $\PP^3$, with one of its linear facets replaced by a cubic hypersurface in such a way that the edges are preserved.
In analogy to the amplituhedron, as discussed in Section \ref{sec:m=2}, $P$ has a genus one curve in its residual arrangement.

\begin{figure}[htbp]
    \centering

    \begin{subfigure}[b]{0.48\textwidth}
        \centering
        \includegraphics[width=\linewidth]{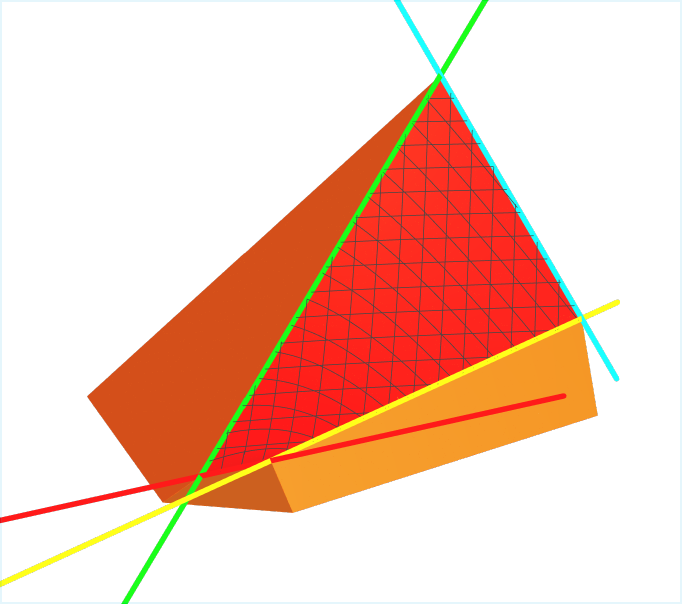}
        \caption{Semi-algebraic set $P$, with a curvy quadrilateral facet on $S$ in red and its boundary lines.}
        \label{fig:3D_DelPezzo}
    \end{subfigure}
    \hfill
    \begin{subfigure}[b]{0.48\textwidth}
        \centering
        \includegraphics[width=\linewidth]{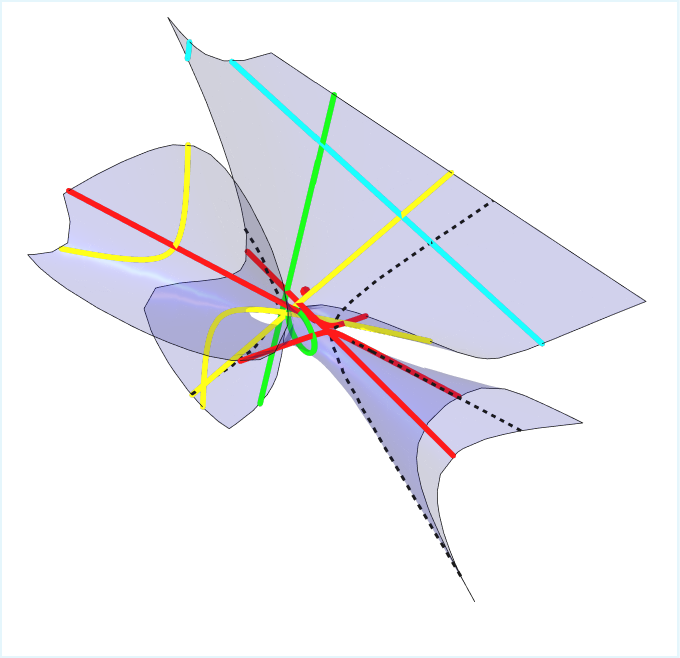}
        \vspace{0cm}
        \caption{Del Pezzo surface $S$ in blue, the quadrilateral face of $P$, bounded by four lines, and the residual elliptic curve in black (dashed).}
        \label{fig:surface}
    \end{subfigure}

    \vspace{1em}

    \begin{subfigure}[t]{0.48\textwidth}
        \centering
        \includegraphics[width=\linewidth]{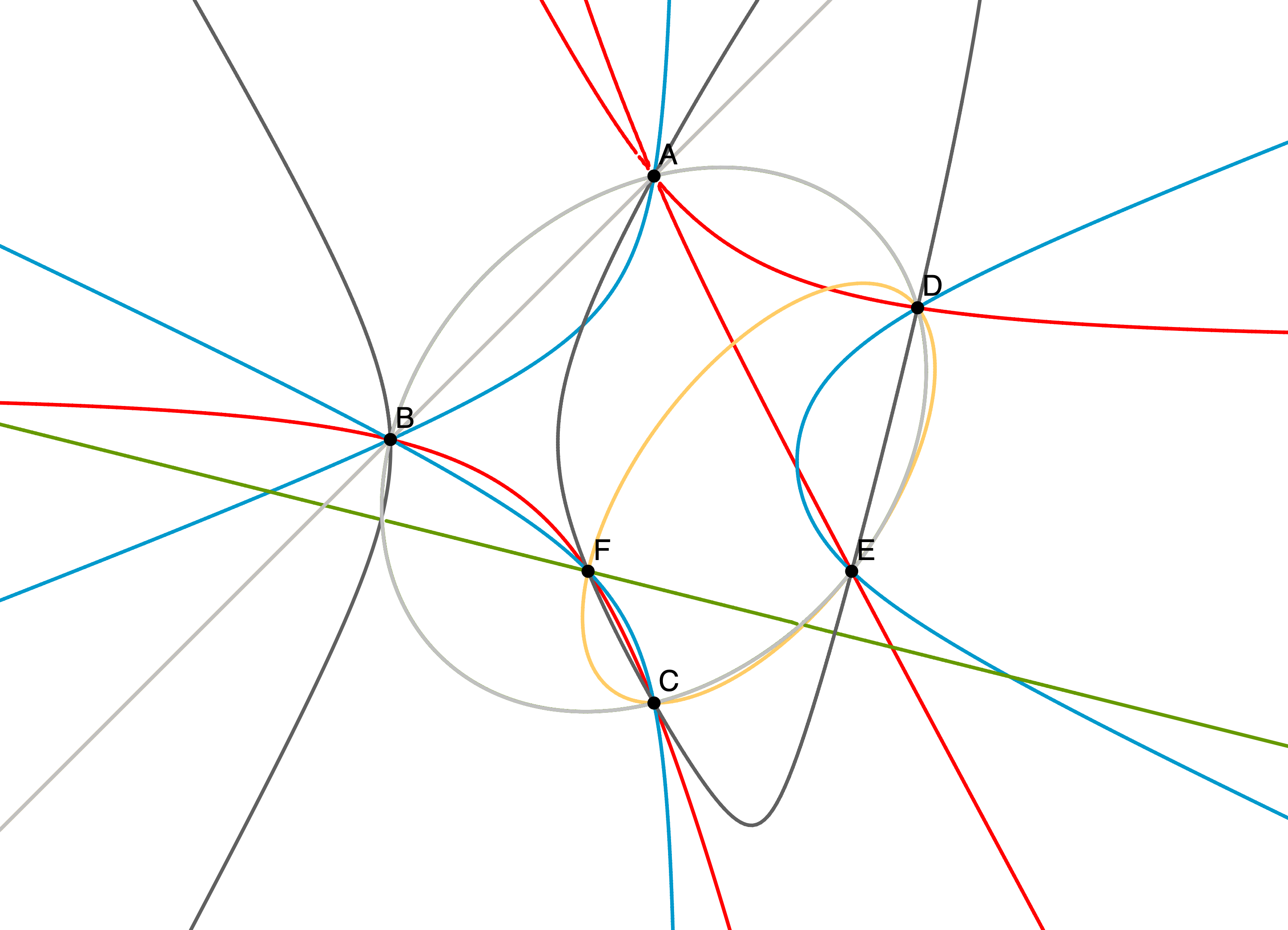}
        \caption{Six points $A$ through $F$ in $\PP^2$ and five cubic curves interpolating these points.}
        \label{fig:delPezzo}
    \end{subfigure}

    \caption{Example of a positive geometry with genus one.}
    \label{fig:all_DelPezzo}
\end{figure}

We start by describing the six facets of $P$.
There are five linear facets contained in the hyperplanes denoted by $L_0,\ldots,L_4$.
The remaining facet of $P$ is contained in a cubic Del Pezzo surface $S$ obtained from a blow-up of six real points in the plane.
The four linear boundary divisors $L_1,\ldots,L_4$ intersect the surface $S$ in four lines, which bound a quadrilateral face of $P$ on $S$, and in four residual conics.
This can be seen in Figure \ref{fig:surface}, where $S$ is the blue-shaded surface, and the quadrilateral is the region with a red point inside.
The semi-algebraic set $P$ has eight more edges: four come from intersecting the facet contained in $L_0$ with the other four linear facets, and the remaining four are the intersections $L_2\cap L_4$, $L_1\cap L_3$, $L_1\cap L_4$ and $L_2\cap L_3$.

As the Del Pezzo surface $S$ is the blow up of $\mathbb{P}^2$ in six points, the combinatorial structure of the semi-algebraic set $P$ can be read off from a plane picture (see Figure \ref{fig:delPezzo}), and we now discuss this in detail.
We pick six points $A$, $B$, $C$, $D$, $E$, and $F$ in $\PP^2$ in generic position with coordinates
\begin{align}
    \begin{array}{lll}
        A = [0:1:1],& B= [-1:0:1], & C=[0:-1:1],\\
        D = [1: 1/2:1],& E = [3/4:-1/2:1],&  F = [-1/4: -1/2:1].
    \end{array}
\end{align}
The blow-up of $\PP^2$ in these six points is a smooth irreducible cubic Del Pezzo surface $S\subset \PP^3$.
This $\PP^3$ is dual to the space of homogeneous cubics vanishing at these six real points. Explicitly, we choose a basis  $c_0,\ldots,c_3\in \C[x,y,z]_3$ with
\begin{align*}
    c_0&=-16 x^{3} - 48 x^{2}y + 118 x y^{2} - 16 x^{2}z + 53 x y z, \\
    c_1&=-32 x^{3} + 10 x^{2}y - 188 x y^{2} + 21 x^{2}z + 53 x z^{2}, \\
    c_2&=84 x^{3} - 172 x^{2}y - 487 x y^{2} - 106 y^{3} + 84 x^{2}z + 106 y z^{2}, \\
    c_3&=-108 x^{3} - 6 x^{2}y + 717 x y^{2} - 161 x^{2}z + 106 x y z - 53 y^{2}z + 53 z^{3}.
\end{align*}
We write $u_0,\ldots,u_3$ for homogeneous coordinates of this $\PP^3$.
Then, the Del Pezzo surface $S$ is the vanishing set of
\begin{align}
    &540u_0^3-381u_0^2u_1+19u_0u_1^2+58u_1^3-402u_0^2u_2+94u_0u_1u_2+27u_1^2u_2-188u_0u_2^2\\
    &-118u_1u_2^2-879u_0^2u_3+84u_0u_1u_3+32u_1^2u_
       3-106u_1u_2u_3+212u_0u_3^2.
\end{align}
In the following we will freely identify a hypersurface with its defining equation.
For the five linear boundary divisors we choose
\begin{align}
    L_0 &= 300u_0 - 900u_1 - 100u_2 - 691u_3,\\
    L_1 &= 267u_0 + 38u_1 + 38u_2 - 76u_3,\\
    L_2 &= 6u_0 + u_1 + 11u_3,\\
    L_3 &= 27u_0 - 8u_1 - 32u_2 - 36u_3,\\
    L_4 &= -257u_0 + 359u_1 - 248u_2 + 16u_3.
\end{align}
If we let $U_3\subset \PP^3$ be the affine open corresponding to points where $u_3\neq 0$, then $P$ is the set of points $q\in U_3$ for which $S(q)\geq 0$ and  $L_i(q)\geq 0$ for $i=0,\ldots,4$, see Figure \ref{fig:3D_DelPezzo}.
We now relate Figure \ref{fig:surface} to Figure \ref{fig:delPezzo} by blowing down to $\PP^2$.
The hyperplanes $L_1$ and $L_2$ correspond to the red and blue cubics with nodes in $A$ and $B$ in Figure \ref{fig:delPezzo}, respectively.
The boundary divisor $L_3$ blows down to the green cubic containing the conic through the points $A,\ldots,E$ and the line through $F$.
The hyperplane $L_4$ corresponds to the reducible cubic in yellow, consisting of the line through $A$ and $B$ and the conic through the remaining four points in the plane.
Therefore, the four lines bounding the facet of $P$ contained in $S$ are obtained as follows: two, red and blue, are the exceptional divisors over the points $A$ and $B$. One, in yellow, is the strict transform of the line through $A$ and $B$, and the fourth one, in green, is the strict transform of the conic through $A,B,C,D,E$.
Finally, the hyperplane $L_0$ corresponds to the elliptic curve through the six points, shown in black (dashed).
Note that all the curves need to be chosen in a way that they do not intersect the region $P'$, bounded by the line through $A$ and $B$ and the part of the conic through $A,\ldots,E$ between $A$ and $B$.
All the relevant plane curves are depicted in Figure \ref{fig:delPezzo}.

We now describe the residual arrangement of $P$. It is combinatorially related to that of the cube: the residual arrangement of a three-dimensional cube consists of three lines, namely the intersection of every pair of boundary divisors containing opposing facets. Similarly, for our set $P$, the residual arrangement  contains two lines, namely $L_1\cap L_2$ and $L_3\cap L_4$ as well as the curve $S\cap L_0$, a smooth curve of genus one, seen in black in Figure \ref{fig:surface}.
Moreover, it also contains the four conics, which arise as irreducible components of the intersections of the four boundary divisors $L_1,\ldots L_4$ with $S$. (Recall, every such intersection decomposes into a conic and a line containing an edge of $P$.)
The residual arrangement is interpolated by a unique homogeneous polynomial of degree $4$ in $u_0,\ldots,u_3$, which is the adjoint $\operatorname{adj}(P)$ of $P$.

We compute that $P$ is a positive geometry in the sense of Definition \ref{def:pos_geom_lam}, with the unique canonical form $\omega_P$ of $P$ given by
\begin{align}
    \omega_P=\frac{\operatorname{adj}(P)}{L_0L_1L_2L_3L_4S}\:  du_0\wedge du_1 \wedge du_2,
\end{align}
where the scaling of $\operatorname{adj}(P)$ is chosen so that all its residues at the vertices of $P$ evaluate to $\pm 1$. In particular, taking the residues of $\omega_P$ along the facets of $P$ recovers their canonical forms. Our code verifying this is available at \cite{zenodo}. By Remark \ref{rem:holomorph}, the uniqueness of the canonical form follows from the fact that there are no non-zero global holomorphic top-degree differential forms on $\mathbb{P}^3$.

Next, let us show that the genus of the pair $(\PP^3,\partial_a P)$ equals one.
First, notice that we constructed the example in such a way that $\partial_a P$ is a simple normal crossing divisor in $\PP^3$.
Therefore, by \cite[Corollary 3.14]{brown-dupont} the genus of the pair $(\PP^3,\partial_a P)$ is bounded from above by the sum of the genera of the intersections of the irreducible components of $\partial_a P$.
Since all boundary divisors are smooth projective varieties (which have vanishing genus), the only non-zero genus contribution stems from two-fold intersections. 
As described above, we have a single elliptic curve $S\cap L_0$ of genus one, and a collection of reducible cubic curves and lines, all of which have vanishing genus.
Therefore, $g(\PP^3,\partial_a P)\leq 1$.
It remains to show that this bound is tight.
Let $D=\bigcup_{i=0}^5D_i=\partial_aP$, with $D_5=S$ and $D_i=L_i$ for $i=0,1,2,3,4$.
Consider the spectral sequence converging to relative cohomology, see e.g. \cite[Equation (12)]{brown-dupont}:
\begin{align}\label{eq:spectral_sequence}
    E^{p,q}_1 = \bigoplus_{|I|=p}H^q(D_I) \quad \Longrightarrow\quad H^{p+q}(\PP^3,D),
\end{align}
where $I\subset \{0,1,2,3,4,5\}$ and $D_I=\bigcap_{i\in I}D_i$. Since $D$ is a normal crossing divisor, this sequence degenerates on the second page, that is all higher differentials between the pages vanish, i.e. $\partial_i=0$ for $i\geq 2$.
For $\ell\notin I$, the obvious inclusion maps $D_{I\cup \ell}\rightarrow D_I$ induce morphisms in cohomology $\delta^q_p:\bigoplus_{|I|=p}H^q(D_I)\rightarrow \bigoplus_{|J|=p+1}H^q(D_J)$.
Therefore, the weight $q$ part of $H^{p+q}(\PP^3,D)$ is computed by $\ker \delta^q_p/\im \delta^q_{p-1}$.
We now show that the weight one part of $H^3(\PP^3,D)$ is at least one-dimensional. So let $q=1$, and thus $p=2$.
Note the $\ker \delta_1^2=H^1(S\cap L_0)\oplus \bigoplus_{1\leq i<j \leq 4}H^1(L_i\cap L_j)$, as the codomain of $\delta_1^2$ is the first cohomology group of a collection of points, which vanishes trivially.
Similarly, the domain of $\delta^1_1$ is the direct sum of the first cohomology groups of the Del Pezzo surface $S$ and the five planes $L_i$.
Since all of these are smooth rational varieties, their first cohomology vanishes. 
We have $\im \delta^1_1=0$.
Using the upper bound for the genus of $(\PP^3,D)$ from above, we get that
\begin{align}\label{eq:weight_one_part}
\operatorname{gr}^W_1(H^3(\PP^3,D))\cong \frac{\ker \delta^2_1}{\im \delta^1_1}\cong H^1(E),
\end{align}
where $E$ is the elliptic curve obtained obtained by intersecting $S\cap L_0$.
This shows that $g(\PP^3,D)=g(E)=1$.

Moreover, the argument above implies that the the map $R:\Omega^3_{\log}(\PP^3\setminus\partial_a P)\rightarrow \operatorname{gr}^W_0H_3(\PP^3,\partial_a P)$, see \cite[Definition 2.4]{brown-dupont}, has a one-dimensional kernel.
We can see that the kernel comes from the genus one curve $E=S\cap L_0$ in the residual arrangement of $P$ as follows.
Consider the $3$-form $\kappa\in \Omega^3_{\log}(\PP^3\setminus\partial_a P)$ given as 
\[
\kappa = \frac{1}{SL_0}du_0\wedge du_1 \wedge du_2,
\]
with a local system of parameters on $\PP^3\setminus\partial_a P$ given by $u_0, u_1$ and $u_2$.
By \cite[Equation (21)]{brown-dupont} we identify $\Omega^3_{\log}(\PP^3\setminus\partial_a P)\cong F^3H^3(\PP^3\setminus\partial_a P)$, whose weight one part, see Equation \eqref{eq:weight_one_part}, is then given by $\kappa$ under this identification.
Then, $R(\kappa)=0$, as the codomain of $R$ is $\operatorname{gr}^W_0H_3(\PP^3,\partial_aP)$, and since the kernel of $R$ is one-dimensional, $\kappa$ is a generator of it.

In \cite[Section 2.7]{brown-dupont} the authors suggest to consider the one-dimensional space $\omega_P+\lambda\kappa$ with $\lambda\in \C$ as the space of candidate canonical forms for $P$.
However, for any $\lambda\neq 0$ the residue of the form $\omega_P+\lambda\kappa$ along the linear divisor $L_0$ has a pole along the residual elliptic curve $E=S\cap L_0$ and thus does not recover the canonical form of the facet of $P$ contained in $L_0$.
Therefore, we must have $\lambda=0$, and thus recover the unique canonical form $\omega_P$.

\begin{remark} \label{rem:final}
Although the pair $(\mathbb{P}^3, \partial_a P)$ has positive genus, one can pick a different ambient variety for the semi-algebraic set $P$ to obtain a genus zero pair. Namely, let $X$ be the blow up of $\mathbb{P}^3$ along the elliptic curve $E$, and $\widetilde Y$ be the strict transform of $\partial_a P$. 
By \cite[Corollary 3.14]{brown-dupont}, we can bound the genus of the pair $(X,\widetilde Y)$ as the sum of the genus of $X$ and the genera of all intersections of the irreducible components of $\widetilde Y$.
To that end, write $\widetilde Y_i$ for the strict transforms of $L_i$ for $i=0,\ldots,4$, and $\widetilde Y_S$ for that of $S$. First, note that $\widetilde{Y}_S \cong S$ and $\widetilde{Y}_0 \cong L_0$, as these two boundary divisors contain the blow up center $E$ as an irreducible subvariety of codimension one. 
Since the divisors $L_i$ for $i=1,\ldots, 4$ intersect $E$ in three points each, their strict transforms are isomorphic to the blow up of $\mathbb{P}^2$ in three points, which is a smooth compact rational variety. 
Then, note that $\widetilde Y_S\cap \widetilde Y_0=\varnothing$, as $S$ and $L_0$ intersect transversally in the blow up center $E$.
Moreover, $\widetilde Y_S\cap \widetilde Y_i\cong S\cap L_i$ for $i=1,\ldots,4$, as the curves $S\cap L_i$ meet $E$ in a finite collection of smooth points.
Similarly, we get $\widetilde Y_i\cap \widetilde Y_j\cong L_i\cap L_j\cong \PP^1$ for $i,j\neq 0$, as these lines do not intersect $E$.
We have just shown that $X$ as well as all intersections of the components of $\widetilde Y$ are smooth rational varieties, hence their genera vanish and $g(X,\widetilde Y)=0$.

The corresponding canonical map would produce a canonical form for the preimage of $P$ under the blow up map. Since the blow up is a morphism from $X$ to $\mathbb{P}^3$, by functoriality of the canonical map the pushforward of this form would then be the canonical form of $P$ in~$\mathbb{P}^3$. 
This suggests that one can sometimes show that a pair $(X, X_{\geq 0})$ is a positive geometry by considering the pair of varieties $(X', Y')$, where $X'$ and $Y'$ are related to $X$ and $\partial_a X_{\geq 0}$ in a natural way. We thank Cl\'ement Dupont for pointing this out. We are optimistic that this phenomenon is rather general, and hope to apply it to the amplituhedron in future work.   
\end{remark}

\bigskip
\noindent\textbf{Acknowledgements.} The authors would like to thank Daniele Agostini, Cl\'ement Dupont, Pavel Galashin, Leonie Kayser, Thomas Lam and Matteo Parisi for helpful discussions. DP is supported by the European Union (ERC, UNIVERSE PLUS, 101118787). Views and opinions expressed are however those of the author(s) only and do not necessarily reflect those of the European Union or the European Research Council Executive Agency. Neither the European Union nor the granting authority can be held responsible for them.

\printbibliography

\noindent{\bf Authors' addresses:}
\medskip

\noindent Joris Koefler, MPI-MiS Leipzig
\hfill {\tt joris.koefler@mis.mpg.de}

\noindent Dmitrii Pavlov, MPI for Physics, Garching bei M\"unchen
\hfill {\tt pavlov@mpp.mpg.de}

\noindent Rainer Sinn, Universit\"at Leipzig
\hfill {\tt rainer.sinn@uni-leipzig.de}

\end{document}